\documentclass[12pt, a4paper]{amsart}
\usepackage[hmargin=32mm, vmargin=27mm, includefoot, twoside]{geometry}

\usepackage[english]{babel}
\usepackage{amsmath,amssymb}
\usepackage{mathrsfs}
\usepackage{dsfont}
\usepackage{lmodern}

\newtheorem{thmintro}{Theorem}

\newtheorem{corintro}[thmintro]{Corollary}
\newtheorem{theorem}{Theorem}[section]
\newtheorem{corollary}[theorem]{Corollary}
\newtheorem{lemma}[theorem]{Lemma}
\newtheorem{prop}[theorem]{Proposition}
\theoremstyle{definition}
\newtheorem{remarkintro}{Remark}
\newtheorem{remark}[theorem]{Remark}

\newcommand{\GG}{\mathbb{G}}
\newcommand{\CC}{\mathbf{C}}
\newcommand{\FF}{\mathbf{F}}
\newcommand{\NN}{\mathbf{N}}
\newcommand{\RR}{\mathbf{R}}
\newcommand{\ZZ}{\mathbf{Z}}
\newcommand{\BB}{\mathcal{B}}
\newcommand{\A}{\mathfrak{A}}
\newcommand{\g}{\mathfrak{g}}
\newcommand{\G}{\mathfrak{G}}
\newcommand{\hh}{\mathfrak{h}}
\newcommand{\B}{\mathfrak{B}}
\newcommand{\N}{\mathfrak{N}}
\newcommand{\T}{\mathfrak{T}}
\newcommand{\U}{\mathfrak{U}}
\newcommand{\UU}{\mathcal{U}}
\newcommand{\inv}{^{-1}}
\newcommand{\co}{\colon\thinspace}
\newcommand{\la}{\langle}
\newcommand{\ra}{\rangle}
\newcommand{\nub}{\mathrm{nub}}

\DeclareMathOperator{\ad}{ad}
\DeclareMathOperator{\Aut}{Aut}
\DeclareMathOperator{\charact}{char}

\DeclareMathOperator{\height}{ht}
\DeclareMathOperator{\re}{re}
\DeclareMathOperator{\im}{im}
\DeclareMathOperator{\Id}{Id}
\DeclareMathOperator{\sign}{sign}
\DeclareMathOperator{\Prop}{P}
\DeclareMathOperator{\con}{con}
\DeclareMathOperator{\Hom}{Hom}

\DeclareMathOperator{\modulo}{mod}
\DeclareMathOperator{\Zalg}{\mathbf{Z}-alg}
\DeclareMathOperator{\Int}{Int}

\begin{document}

\renewcommand{\proofname}{{\bf Proof}}

\title[Abstract simplicity of locally compact Kac--Moody groups]{Abstract simplicity of locally compact Kac--Moody groups}

\author[T.~Marquis]{Timoth\'ee \textsc{Marquis}$^*$}
\address{UCL, 1348 Louvain-la-Neuve, Belgium}
\email{timothee.marquis@uclouvain.be}
\thanks{$^*$F.R.S.-FNRS Research Fellow}

%
\begin{abstract}
In this paper, we establish that complete Kac--Moody groups over finite fields are abstractly simple. The proof makes an essential use of Mathieu--Rousseau's construction of complete Kac--Moody groups over fields. This construction has the advantage that both real and imaginary root spaces of the Lie algebra lift to root subgroups over arbitrary fields. A key point in our proof is the fact, of independent interest, that both real and imaginary root subgroups are contracted by conjugation of positive powers of suitable Weyl group elements.
\end{abstract}

\maketitle

\section{Introduction}
Let $A=(A_{ij})_{1\leq i,j\leq n}$ be a generalised Cartan matrix and let $\G=\G_A$ denote the associated Kac--Moody--Tits functor of simply connected type, as defined by J.~Tits (\cite{Tits87}). The value of $\G$ over a field $k$ is usually called a {\bf minimal Kac--Moody group} of type $A$ over $k$. This terminology is justified by the existence of larger groups associated with the same data, usually called {\bf maximal} or {\bf complete Kac--Moody groups}, and which are completions of $\G(k)$ with respect to some suitable topology. 
One of them, introduced in \cite{ReRo}, and which we will temporarily denote by $\hat{\G}_A(k)$, is a totally disconnected topological group. It is moreover locally compact provided $k$ is finite, and non-discrete (hence uncountable) as soon as $A$ is not of finite type.

The question whether $\hat{\G}_A(k)$ is (abstractly) simple for $A$ indecomposable and $k$ arbitrary is very natural and was explicitly addressed by J.~Tits \cite{Tits89}. Abstract simplicity results for $\hat{\G}_A(k)$ over fields of characteristic $0$ were first obtained in an unpublished note by R.~Moody (\cite{Moodyunpublished}). Moody's proof has been recently generalised by G.~Rousseau (\cite[Théorème~6.19]{Rousseau}) who extended Moody's result to fields $k$ of positive characteristic $p$ that are not algebraic over $\FF_p$. The abstract simplicity of $\hat{\G}_A(k)$ when $k$ is a finite field was shown in \cite{CER} in some important special cases, including groups of $2$-spherical type over fields of order at least $4$, as well as some other hyperbolic types under additional restrictions on the order of the ground field.

In this paper, we establish the abstract simplicity of complete Kac--Moody groups $\hat{\G}_A(k)$ of indecomposable type over arbitrary finite fields, without any restriction. Our proof relies on an approach which is completely different from the one used in \cite{CER}.

\begin{thmintro}\label{mainthmintro}
Let $\hat{\G}_A(\FF_q)$ be a complete Kac--Moody group over a finite field $\FF_q$, with generalised Cartan matrix $A$. Assume that $A$ is indecomposable of indefinite type. Then $\hat{\G}_A(\FF_q)$ is abstractly simple.
\end{thmintro}
As it turns out, it does not matter which completion of $\G_A(\FF_q)$ we are considering, see Theorem~\ref{thmintro simple pma} and Remark~\ref{remintro other completions} below.

After completion of this work, I was informed by Bertrand Rémy that, in a recent joint work \cite{RCap} with I.~Capdeboscq, they obtained independently a special case of this theorem, namely the abstract simplicity over finite fields of order at least $4$ and of characteristic $p$ in case $p$ is greater than $M=\max_{i\neq j}{|A_{ij}|}$. Their approach is similar to the one used in \cite{CER}.

Note that the topological simplicity of $\hat{\G}_A(\FF_q)$ (that is, the absence of nontrivial closed normal subgroups), which we will use in our proof of Theorem~\ref{mainthmintro}, was previously established by B.~Rémy when $q>3$ (see \cite[Theorem~2.A.1]{Remytopolsimple}); the tiniest finite fields were later covered by P-E.~Caprace and B.~Rémy (see \cite[Proposition~11]{CaRe}).

Note also that for incomplete groups, abstract simplicity fails in general since groups of affine type admit numerous congruence quotients. However, it has been shown by P-E.~Caprace and B.~Rémy (\cite{CaRe}) that $\G_A(\FF_q)$ is abstractly simple provided $A$ is indecomposable, $q>n>2$ and $A$ is not of affine type. They also recently covered the rank $2$ case for matrices $A$ of the form $A=(\begin{smallmatrix}2&-m\\ -1&2\end{smallmatrix})$ with $m>4$ (see \cite[Theorem~2]{CaRerk2}). 

\medskip

As mentioned at the beginning of this introduction, different completions of $\G(k)$ were considered in the literature, and therefore all deserve the name of ``complete Kac--Moody groups". We now proceed to review them briefly.

Essentially three such completions have been constructed so far, from very different points of view. The first construction, due to B.~Rémy and M.~Ronan (\cite{ReRo}), is the one we considered above. It is the completion of the image of $\G(k)$ in the automorphism group $\Aut(X_+)$ of its associated positive building $X_+$, where $\Aut(X_+)$ is equipped with the compact-open topology. For the rest of this paper, we will denote this group by $\G^{rr}(k)$, so that $\hat{\G}(k)=\G^{rr}(k)$ in our previous notation. To avoid taking a quotient of $\G(k)$, a variant of this group has also been considered by P-E.~Caprace and B.~Rémy (\cite[Section~1.2]{CaRe}). This latter group, here denoted $\G^{crr}(k)$, contains $\G(k)$ as a dense subgroup and admits $\G^{rr}(k)$ as a quotient.

The second construction, due to L.~Carbone and H.~Garland (\cite{CarboneGarland}), associates to a regular dominant integral weight $\lambda$ the completion, here denoted $\G^{cg\lambda}(k)$, of $\G(k)$ for the so-called weight topology. 

The third construction, of which we will make an essential use, was first introduced by O.~Mathieu (\cite[XVIII \S 2]{M88a}, \cite{M88b}, \cite[I and II]{M89}) and further developed by G.~Rousseau (\cite{Rousseau}). It is more algebraic and closer in spirit to the construction of $\G$. In fact, one gets a group functor over the category of $\ZZ$-algebras, which we will subsequently denote by $\G^{pma}$. As noted in \cite[3.20]{Rousseau}, this functor is a generalisation of the complete Kac--Moody group over $\CC$ constructed by S.~Kumar (\cite[Section~6.1.6]{Kumar}). Note that in this case the closure $\overline{\G(k)}$ of $\G(k)$ in $\G^{pma}(k)$ need not be the whole of $\G^{pma}(k)$. However, $\overline{\G(k)}=\G^{pma}(k)$ as soon as the characteristic of $k$ is zero or greater than the maximum $M$ (in absolute value) of the non-diagonal entries of $A$ (see \cite[Proposition~6.11]{Rousseau}).

These three constructions are strongly related, and hopefully equivalent. In particular, they all possess refined Tits systems whose associated building is the positive building $X_+$ of $\G(k)$ (with possibly different apartment systems). Moreover, there are natural continuous group homomorphisms $\overline{\G(k)}\to \G^{cg\lambda}(k)$ and $\G^{cg\lambda}(k)\to \G^{crr}(k)$ extending the identity on $\G(k)$ (see \cite[6.3]{Rousseau}). Their composition $\phi\co\overline{\G(k)}\to \G^{cg\lambda}(k)\to \G^{crr}(k)$
is an isomorphism of topological groups in many cases (see \cite[Théorème 6.12]{Rousseau}) and conjecturally in all cases.

If $G$ is either $\G^{pma}(k)$ or $\overline{\G(k)}$ or $\G^{cg\lambda}(k)$ or $\G^{crr}(k)$, we let $Z'(G)$ denote the kernel of the $G$-action on $X_+$. As mentioned in Remark~\ref{remintro other completions} below, Theorem~\ref{mainthmintro} immediately implies the abstract simplicity of $G/Z'(G)$ whenever $G$ is one of $\overline{\G(k)}$ or $\G^{cg\lambda}(k)$ or $\G^{crr}(k)$ (and $k$ is finite). As pointed out to me by Pierre-Emmanuel Caprace, our arguments in fact also imply the abstract simplicity of $\G^{pma}(k)/Z'(\G^{pma}(k))$, even when $\overline{\G(k)}\neq \G^{pma}(k)$:

\begin{thmintro}\label{thmintro simple pma}
Assume that the generalised Cartan matrix $A$ is indecomposable of indefinite type. Then $\G_A^{pma}(\FF_q)/Z'(\G_A^{pma}(\FF_q))$ is abstractly simple over any finite field $\FF_q$.
\end{thmintro}
Note that even the topological simplicity of $\G_A^{pma}(\FF_q)/Z'(\G_A^{pma}(\FF_q))$ was not previously known in full generality (see \cite[Lemme 6.14 and Proposition 6.16]{Rousseau} for known results). 

While the construction of Rémy--Ronan is more geometric in nature, the construction of Mathieu--Rousseau is purely algebraic and hence \emph{a priori} more suitable to establish algebraic properties of complete Kac--Moody groups. The present paper is a good illustration of this idea, and we hope it provides a good motivation for studying these ``algebraic completions" further.

\begin{remarkintro}\label{remintro other completions}
When the field $k$ is finite, the several group homomorphisms $\overline{\G(k)}\to \G^{cg\lambda}(k)\to \G^{crr}(k)\to \G^{rr}(k)\leq\Aut(X_+)$ are all surjective (see \cite[6.3]{Rousseau}), and if $G$ is either $\overline{\G(k)}$ or $\G^{cg\lambda}(k)$ or $\G^{crr}(k)$, the effective quotient of $G$ by the kernel $Z'(G)$ of its action on $X_+$ coincides with $\G^{rr}(k)$. If moreover the characteristic $p$ of $k$ is greater than the maximum $M$ (in absolute value) of the non-diagonal entries of $A$, one has $\overline{\G(k)}=\G^{pma}(k)$, and thus in that case there is only one simple group $G/Z'(G)$. Hence Theorem~\ref{thmintro simple pma} is a consequence of Theorem~\ref{mainthmintro} when $p>M$. If $p\leq M$, it is possible that the effective quotient of $\G^{pma}(k)$ inside $\Aut(X_+)$ properly contains $\G^{rr}(k)$ (see Corollary~\ref{corintro non density} below). When this happens, Theorem~\ref{thmintro simple pma} thus asserts the abstract simplicity of a different group than the one considered in Theorem~\ref{mainthmintro}.

Finally, we notice that, although we assumed the Kac--Moody groups to be of simply connected type to simplify the notations, the results remain valid for an arbitrary Kac--Moody root datum (see Remark~\ref{rem blabla}). 
\end{remarkintro}

\medskip

Along the proof of Theorems~\ref{mainthmintro} and \ref{thmintro simple pma}, we establish other results of independent interest, which we now proceed to describe.

Let $k$ be an arbitrary field. Fix a realisation of the generalised Cartan matrix $A=(a_{ij})_{1\leq i,j\leq n}$ as in \cite[\S 1.1]{Kac}. Let $Q=\sum_{i=1}^n{\ZZ\alpha_i}$ be the associated root lattice, where $\alpha_1,\dots,\alpha_n$ are the simple roots. Let also $\Delta$ (respectively, $\Delta_{\pm}$) be the set of roots (respectively, positive/negative roots), so that $\Delta=\Delta_+\sqcup\Delta_-$. Write also $\Delta^{\re}$ and $\Delta^{\im}$ (respectively, $\Delta_+^{\re}$ and $\Delta_+^{\im}$) for the set of (positive) real and imaginary roots. 

Recall that a subset $\Psi$ of $\Delta$ is {\bf closed} if $\alpha+\beta\in\Psi$ whenever $\alpha,\beta\in\Psi$ and $\alpha+\beta\in\Delta$. For a closed subset $\Psi$ of $\Delta_+$, define the sub-group scheme $\U^{ma}_{\Psi}$ of $\G^{pma}$ as in \cite[3.1]{Rousseau}. Set $\U^{ma+}=\U^{ma}_{\Delta_+}$. One can then define {\bf root groups} $\U^{ma}_{(\alpha)}$ in $\U^{ma+}$ by setting $\U^{ma}_{(\alpha)}:=\U^{ma}_{\{\alpha\}}$ for $\alpha\in\Delta_+^{\re}$ and $\U^{ma}_{(\alpha)}:=\U^{ma}_{\NN^*\alpha}$ for $\alpha\in\Delta_+^{\im}$, where $\NN^*=\NN\setminus\{0\}$.  

We also let $\B^+$, $\U^+$, $\N$ and $\T$ denote, as in \cite[1.6]{Rousseau}, the sub-functors of $\G=\G_A$ such that over $k$, $(\B^+(k)=\U^+(k)\rtimes\T(k),\N(k))$ is the canonical positive BN-pair attached to $\G(k)$, and $\N(k)/\T(k)\cong W$, where $W=W(A)$ is the Coxeter group attached to $A$. We fix once for all a section $W\cong \N(k)/\T(k)\to \N(k):w\mapsto\overline{w}$. Note that $\N$ can be viewed as a sub-functor of $G^{pma}$ (see \cite[3.12, Remarque 1]{Rousseau}). 

Finally, given a topological group $H$ and an element $a\in H$, we define the {\bf contraction group} $\con^H(a)$, or simply $\con(a)$, as the set of elements $g\in H$ such that $a^nga^{-n}\stackrel{n\to\infty}{\rightarrow} 1$. Note then that for any $a\in\overline{\G(k)}\subseteq\G^{pma}(k)$, one has $\varphi(\con^{\G^{pma}(k)}(a)\cap\overline{\G(k)})\subseteq \con^{\G^{rr}(k)}(\varphi(a))$, where we denote by $\varphi$ the composition $\overline{\G(k)}\stackrel{\phi}{\to} \G^{crr}(k)\to\G^{rr}(k)$.

\begin{thmintro}\label{thmBintro}
Let $k$ be an arbitrary field. 
\begin{enumerate}
\item
Let $\omega\in W$ and let $\Psi\subseteq\Delta_+$ be a closed set of positive roots such that $\omega\Psi\subseteq\Delta_+$. Then  $\overline{\omega}\U^{ma}_\Psi\overline{\omega}\thinspace\inv=\U^{ma}_{\omega\Psi}$.
\item
Let $\omega\in W$ and $\alpha\in\Delta_+$ be such that $\omega^l\alpha$ is positive and different from $\alpha$ for all positive integers $l$. Then $\U^{ma}_{(\alpha)}\subseteq \con^{\G^{pma}(k)}(\overline{\omega})$. In particular $\varphi(\U^{ma}_{(\alpha)}\cap\overline{\G(k)})\subseteq \con^{\G^{rr}(k)}(\overline{\omega})$.
\item
Assume that $A$ is of indefinite type. Then there exists some $\omega\in W$ such that $\U^{ma}_{(\alpha)}\subseteq \con^{\G^{pma}(k)}(\overline{\omega})\cup \con^{\G^{pma}(k)}(\overline{\omega}\thinspace\inv)$ for all $\alpha\in\Delta_+$. Hence root groups (associated to both real and imaginary roots) are contracted.
\end{enumerate}
\end{thmintro}

The proof of Theorem~\ref{thmBintro} can be found at the end of Section~\ref{section contraction}. The idea to prove Theorem~\ref{mainthmintro} once Theorem~\ref{thmBintro} is established is the following. We let $a\in \N(\FF_q)$ be such that $\U^{ma}_{(\alpha)}(\FF_q)\subseteq \con^{\G^{pma}(\FF_q)}(a)\cup \con^{\G^{pma}(\FF_q)}(a\inv)$ for all $\alpha\in\Delta_+$, as in Theorem~\ref{thmBintro}~(3). We deduce that $\U^{rr+}(\FF_q)$ is contained in the subgroup generated by the closures of $\con^{\G^{rr}(\FF_q)}(a^{\pm 1})$. Now, as the topological simplicity of $\G^{rr}(\FF_q)$ is known, it suffices to consider a dense normal subgroup $K$ of $\G^{rr}(\FF_q)$. We can then conclude by invoking the following result of Caprace--Reid--Willis (\cite[Theorem~1.1]{Titscore}):

\begin{theorem}\label{thm CRW}
Let $G$ be a totally disconnected locally compact group and let $K$ be a dense normal subgroup of $G$. Then $K$ contains the closure in $G$ of $\con(g)$ for any $g\in G$.
\end{theorem}

The proof of Theorem~\ref{thmintro simple pma} follows the exact same lines, except that in this case the topological simplicity of the group is not known in full generality, and we need one more argument to establish it.

\smallskip

We also point out that Theorem~\ref{thmBintro} has another application, concerning the existence of non-closed contraction groups in complete Kac--Moody groups of non-affine type. Recall that in simple algebraic groups over local fields, contraction groups are always closed (they are in fact either trivial, or coincide with the unipotent radical of some parabolic subgroup). In particular they are closed in a complete Kac--Moody group $G$ over a finite field as soon as the defining generalised Cartan matrix $A$ is of non-twisted affine type. It has been shown in \cite{BRBcontraction} that, on the other hand, if $A$ is indecomposable non spherical, non affine and of size at least $3$, then the contraction group $\con(a)$ of some element $a\in G$ must be non-closed. The following result shows that this also holds when $A$ is indecomposable non spherical, non affine and of size $2$. 

\begin{thmintro}\label{thmintro closed contraction groups}
Let $A$ denote an $n\times n$ generalised Cartan matrix of indecomposable indefinite type, let $W=W(A)$ be the associated Weyl group, and let $w=s_1\dots s_n$ denote the Coxeter element of $W$. Let also $G$ be one the complete Kac--Moody groups $\G_A^{rr}(\FF_q)$ or $\G_A^{pma}(\FF_q)$ of simply connected type. Then the contraction group $\con^G(w)$ is not closed in $G$, unless maybe if $G=\G_A^{rr}(\FF_q)$ and $\U^{ma}_{\Delta_+^{\im}}(\FF_q)\cap \overline{\G(\FF_q)}$ is contained in the kernel of $\varphi$.
\end{thmintro}

Finally, here is a last application of our results concerning isomorphism classes of Kac--Moody groups and their completions. While over infinite fields, it is known that two minimal Kac--Moody groups can be isomorphic only if their ground fields are isomorphic and their underlying generalised Cartan matrices coincide up to a row-column permutation (see \cite[Theorem A]{thesePE}), this fails to be true over finite fields. Indeed, over a given finite field, two minimal Kac--Moody groups associated with two different generalised Cartan matrices of size $2$ can be isomorphic, as noticed in \cite[Lemma 4.3]{thesePE}. The following result shows that, however, the corresponding Mathieu--Rousseau completions should not be expected to be isomorphic as topological groups.

\begin{thmintro}\label{thmintro isom KM}
Let $k=\FF_q$ be a finite field with $\charact k\neq 2$. Then there exist minimal Kac--Moody groups $G_1=\G_{A_1}(\FF_q)$ and $G_2=\G_{A_2}(\FF_q)$ over $\FF_q$ associated to $2\times 2$ generalised Cartan matrices $A_1$, $A_2$, such that $G_1$ and $G_2$ are isomorphic as abstract groups, but their Mathieu--Rousseau completions $\G^{pma}_{A_1}(\FF_q)$ and  $\G^{pma}_{A_2}(\FF_q)$ are not isomorphic as topological groups. 
\end{thmintro}

This surprising result provides in particular the first known families of examples over arbitrary finite fields (of characteristic at least $3$) of minimal Kac--Moody groups that are not dense in their Mathieu--Rousseau completion (up to now, the only known such family was given over $\FF_2$ in \cite[6.10]{Rousseau}).

\begin{corintro}\label{corintro non density}
Let $k=\FF_q$ be a finite field with $\charact k\neq 2$. Let $A=(\begin{smallmatrix}2&-m \\ -n&2\end{smallmatrix})$ be a generalised Cartan matrix with $m,n>2$ and assume that  $m\equiv n \equiv 2 \ (\modulo q-1)$. Then the minimal Kac--Moody group $\G_A(\FF_q)$ of simply connected type is not dense in its Mathieu--Rousseau completion $\G^{pma}_A(\FF_q)$.
\end{corintro}

The proof of these statements will be given in Section~\ref{section proof corollaries}. 

\medskip

The paper is organised as follows. We first fix some notations and provide an outline of the construction of Mathieu--Rousseau completions in Section~\ref{section notation}. We next prove some preliminary results about the Coxeter group $W$ and the set of roots $\Delta$ in Section~\ref{section W and Delta}. We then use these results to prove a more precise version of Theorem~\ref{thmBintro} in Section~\ref{section contraction}. We establish its consequences in Section~\ref{section proof corollaries}, and we conclude the proof of Theorems~\ref{mainthmintro} and \ref{thmintro simple pma} in Section~\ref{section proof thmA}.

\subsection*{Acknowledgement.}
I am very grateful to Pierre-Emmanuel Caprace for proposing this problem to me in the first place, as well as for numerous helpful comments. I would also like to thank the anonymous referee for his/her useful comments.

\section{Preliminaries}\label{section notation}
\subsection{Notations}
Throughout this paper, we write $\NN^*$ (resp. $\ZZ^*$) for the set of nonzero natural numbers (resp. nonzero integers). 

For the rest of this paper, $k$ denotes an arbitrary field and $A=(a_{ij})_{i,j\in I}$ denotes a generalised Cartan matrix indexed by $I=\{1,\dots,n\}$. We fix a realisation $(\hh,\Pi,\Pi^{\vee})$ of $A$ as in \cite[\S 1.1]{Kac}. We then keep all notations from the introduction. In particular, $\Delta$ is the corresponding set of roots and $\Pi=\{\alpha_1,\dots,\alpha_n\}$ (resp. $\Pi^{\vee}=\{\alpha_1^{\vee},\dots,\alpha_n^{\vee}\}$) the set of simple roots (resp. coroots). For $\alpha\in\Delta$, we denote by $\height(\alpha)$ its height.

Recall the definition of the Tits functor $\G=\G_A$ and of its sub-functors $\B^+$, $\U^+$, $\N$ and $\T$. Again, $\G^{rr}(k)$ denotes the Rémy--Ronan completion of $\G(k)$ and $\U^{rr+}(k)$ the completion in $\G^{rr}(k)$ of $\U^+(k)$, so that $(\U^{rr+}(k)\rtimes\T(k), \N(k))$ is a BN-pair for $\G^{rr}(k)$ (see \cite[Proposition 1]{CaRe}). We will give more details about the Mathieu--Rousseau completion $\G^{pma}(k)$ of $\G(k)$ in Section~\ref{subsection outline MR completion} below.

As before, $W=W(A)\cong\N(k)/\T(k)$ is the Coxeter group associated to $A$, with generating set $S=\{s_1,\dots,s_n\}$ such that $s_i(\alpha_j)=\alpha_j-a_{ij}\alpha_i$ for all $i,j\in I$, and we fix a section $W\cong \N(k)/\T(k)\to \N(k):w\mapsto\overline{w}$.

Finally, to avoid cumbersome notation, we will write $\con(a)$ for both contraction groups $\con^{\G^{pma}(k)}(a)$ and $\con^{\G^{rr}(k)}(a)$, as $k$ is fixed and as it will be always clear in which group we are working.

\subsection{The Mathieu--Rousseau completion}\label{subsection outline MR completion}
We now outline the construction of the Mathieu--Rousseau completion of $\G$ and give its basic properties, as it will play an important role in what follows. The general reference for this section is \cite{Rousseau}.

\subsubsection*{Some notations}
Let $\Lambda^{\vee}$ be the free $\ZZ$-module generated by $\Pi^{\vee}$, and let $\Lambda$ be its $\ZZ$-dual, which we view as a $\ZZ$-form of the dual $\hh^*$. In particular, $\Lambda$ contains $\Pi$. Then, as we are considering a Tits functor $\G_A$ of simply connected type, the torus $\T(k)=\T_{\Lambda}(k)=\Hom_{gr}(\Lambda,k^\times)$ is generated by $\{r^{h} \ | \ r\in k^\times, \ h\in\Pi^{\vee}\}$, where $$r^{h}\co \Lambda\to k^\times: \lambda\mapsto r^{\la \lambda,h\ra}.$$

Let $\g$ denote the Kac--Moody algebra of $\G$ with root space decomposition $\g=\hh\oplus\bigoplus_{\alpha\in\Delta}{\g_{\alpha}}$, and let $e_1,\dots,e_n$ and $f_1,\dots,f_n$ be the corresponding Chevalley generators, so that $\g_{\alpha_i}=\CC e_i$ and $\g_{-\alpha_i}=\CC f_i$ for all $i\in I$. 
Let also $\UU$ denote the $\ZZ$-form of the enveloping algebra $\UU_{\CC}(\g)$ of $\g$ introduced by J.~Tits (see e.g. \cite[Section~2]{Rousseau}): this is a $\ZZ$-bialgebra graded by $Q:=\bigoplus_{i\in I}{\ZZ\alpha_i}$ and containing the elements $e_i^{(l)}:=e_i^l/l!$ and $f_i^{(l)}:=f_i^l/l!$ ($l\in\NN$, $i\in I$). We write $\UU_{\alpha}$ for the weight space corresponding to $\alpha\in Q$. The $W$-action on $\Delta$ induces a $W$-action on $\UU_{\CC}(\g)$ with $s_i$ ($i\in I$) acting as $$s_i^*=\exp(\ad e_i)\exp(\ad f_i)\exp(\ad e_i)\in \Aut(\UU_{\CC}(\g)).$$
This $W$-action preserves $\UU$, and given $\alpha\in\Delta^{\re}$ such that $\alpha=w\alpha_i$ for some $w\in W$ and $i\in I$, the element $e_{\alpha}=w^*e_i$ is well defined (up to a choice of sign) and is a $\ZZ$-basis for $\g_{\alpha\ZZ}:=\g_{\alpha}\cap \UU$. In particular, we may choose $e_{-\alpha_i}:=f_i$ as a basis for $\g_{-\alpha_i\ZZ}$. For a ring $R$, we also set $\g_{\alpha R}:=\g_{\alpha\ZZ}\otimes_{\ZZ}R$.

For a closed set $\Psi\subseteq\Delta$, we define the $\ZZ$-subalgebra $\UU(\Psi)$ of $\UU$ generated by all $\UU^{\alpha}:=\UU_{\CC}(\oplus_{n\geq 1}\g_{n\alpha})\cap\UU$ for $\alpha\in\Psi$. If in addition $\Psi\subseteq w(\Delta_+)$ for some $w\in W$, we may also define the completion $\widehat{\UU}_R(\Psi)$ of $\UU(\Psi)$ over any ring $R$ as $$\widehat{\UU}_R(\Psi)=\prod_{\alpha\in w.Q_+}{(\UU(\Psi)_{\alpha}\otimes_{\ZZ}R)},$$
where $Q_+:=\bigoplus_{i\in I}{\NN\alpha_i}$ and $\UU(\Psi)_{\alpha}=\UU(\Psi)\cap\UU_\alpha$.

\subsubsection*{Pro-unipotent groups}
The first step in the construction of $\G^{pma}$ is to define for each closed set $\Psi\subseteq\Delta_+$ of positive roots the affine group scheme $\U^{ma}_{\Psi}$ (which we view as a group functor) whose algebra is the restricted dual $$\ZZ[\U^{ma}_{\Psi}]:=\bigoplus_{\alpha\in\NN\Psi}{\UU(\Psi)^*_{\alpha}}$$ of $\UU(\Psi)$. In other words,
 $$\U^{ma}_{\Psi}(R)=\Hom_{\Zalg}(\ZZ[\U^{ma}_{\Psi}],R) \quad \textrm{for any ring $R$.}$$
One can then define {\bf root groups} $\U^{ma}_{(\alpha)}$ in $\U^{ma+}:=\U^{ma}_{\Delta_+}$ by setting $\U^{ma}_{(\alpha)}=\U_{\alpha}:=\U^{ma}_{\{\alpha\}}$ for $\alpha\in\Delta_+^{\re}$ and $\U^{ma}_{(\alpha)}:=\U^{ma}_{\NN^*\alpha}$ for $\alpha\in\Delta_+^{\im}$. For $\alpha\in\Delta^{\re}_+$, one can define similarly the root group $\U_{-\alpha}=\U^{ma}_{\{-\alpha\}}$ as above, with $\Psi$ replaced by $\{-\alpha\}$. In other words, for each $\alpha\in\Delta^{\re}$, the real root group $\U_{\alpha}$ is isomorphic to the additive group scheme $\GG_a$ by $$x_{\alpha}\co\GG_a(R)\stackrel{\sim}{\to}\U_{\alpha}(R):r\mapsto \exp(re_{\alpha})$$ for any ring $R$. Note that, identifying $\{\U_{\alpha}(R) \ | \ \alpha\in\Delta^{\re}\}$ with the root group datum of $\G(R)$, the element $\overline{s_i}\in\N(R)$ lifting $s_i\in W$ may then be chosen as 
$$\overline{s_i}=x_{\alpha_i}(1)x_{-\alpha_i}(1)x_{\alpha_i}(1)=\exp(e_i)\exp(f_i)\exp(e_i).$$

The group functor $\U^{ma}_{\Psi}$ admits a nice description in terms of root groups, which we now briefly review. For each $x\in\g_{\alpha\ZZ}$, $\alpha\in\Delta_+$, G.~Rousseau makes a choice of an {\bf exponential sequence}, namely of a sequence $(x^{[n]})_{n\in\NN}$ where $x^{[0]}=1$, $x^{[1]}=x$, and $x^{[n]}\in\UU_{n\alpha}$ is such that $x^{[n]}-x^n/n!$ has filtration less than $n$ in $\UU_{\CC}(\g)$ for each $n\in\NN$, and which satisfies some additional compatibility condition with the bialgebra structure on $\UU$ (see \cite[Propositions~2.4 and 2.7]{Rousseau}). Such an exponential sequence for $x$ is unique up to modifying each $x^{[n]}$, $n\geq 2$, by an element of $\g_{n\alpha\ZZ}$. For a ring $R$ and an element $\lambda\in R$, one can then define the {\bf twisted exponential} $$[\exp]\lambda x:=\sum_{n\geq 0}{\lambda^nx^{[n]}}\in\widehat{\UU}_R(\Delta_+).$$
Note that for $\alpha$ a real root, one can take the usual exponential. For each $\alpha\in\Delta_+$, let $\BB_{\alpha}$ be a $\ZZ$-basis of $\g_{\alpha\ZZ}$. For $\alpha\in\Delta^{\re}$, we choose $\BB_{\alpha}=\{e_{\alpha}\}$. Finally, for a closed subset $\Psi\subseteq \Delta_+$, set $\BB_\Psi=\bigcup_{\alpha\in\Psi}{\BB_{\alpha}}$. Here is the announced description of $\U^{ma}_{\Psi}$:

\begin{prop}[{\cite[Proposition~3.2]{Rousseau}}] \label{prop formal sum Rousseau}
Let $\Psi\subseteq\Delta_+$ be closed and let $R$ be a ring. Then $\U^{ma}_{\Psi}(R)$ can be identified to the multiplicative subgroup of $\widehat{\UU}_R(\Psi)$ consisting of the products $$\prod_{x\in\BB_{\Psi}}{[\exp]\lambda_xx}$$ for $\lambda_x\in R$, where the product is taken in any (arbitrary) chosen order on $\BB_\Psi$. The expression of an element of $\U^{ma}_{\Psi}(R)$ in the form of such a product is unique.
\end{prop}

A consequence of this proposition which we will use later on is the following (see \cite[Lemme~3.3]{Rousseau}).

\begin{lemma}\label{lemma decomp Rousseau}
Let $\Psi'\subseteq\Psi\subseteq\Delta_+$ be closed subsets of roots. Then $\U^{ma}_{\Psi'}$ is a closed subgroup of $\U^{ma}_{\Psi}$. Moreover, if $\Psi\setminus\Psi'$ is closed as well, then there is a unique decomposition $\U^{ma}_\Psi=\U^{ma}_{\Psi'}.\U^{ma}_{\Psi\setminus\Psi'}$.
\end{lemma}

\subsubsection*{Minimal parabolics}
The next step in the construction of $\G^{pma}$ is to define, for each $i\in I$, the {\bf minimal parabolic subgroup} $\B_i^{ma+}$ of type $i$ as the semi-direct product of $\U^{ma}_{\Delta_+\setminus\{\alpha_i\}}$ with the unique connected affine algebraic group $\A^{\Lambda}_{i}$ associated to the Kac--Moody root datum $(\{1\},(2),\Lambda, \{\alpha_i\},\{\alpha_i^{\vee}\})$ (see \cite[Theorem~10.1.1]{Springer}). Note that $\A^{\Lambda}_{i}$ contains $\T$, $\U_{\alpha_i}$ and $\U_{-\alpha_i}$ as closed subgroups and is generated by them. To define this semi-direct product, it is thus sufficient to describe for each ring $R$ conjugation actions of $\T(R)=\Hom_{gr}(\Lambda,R^\times)$ and $\U_{\alpha}(R)=\{\exp(re_{\alpha}) \ | \ r\in R \}$ on $\U^{ma}_{\Delta_+\setminus\{\alpha_i\}}$, for $\alpha\in\{\pm\alpha_i\}$. For $t\in \T(R)$, this is defined using Proposition~\ref{prop formal sum Rousseau} by
$$\Int(t)\cdot [\exp]\lambda x=[\exp]t(\gamma)\lambda x\quad\textrm{if $x\in\g_{\gamma R}$}.$$
For $\alpha\in\{\pm\alpha_i\}$, we set 
$$\Int(\exp(e_\alpha))(z)=\sum_{m\geq 0}{(\ad(e_{\alpha})^m/m!)(z)}$$
for all $z\in\U^{ma}_{\Delta_+\setminus\{\alpha_i\}}(R)$, where $\U^{ma}_{\Delta_+\setminus\{\alpha_i\}}(R)$ is viewed as a subset of either $\widehat{\UU}_R(\Delta_+)$ or $\widehat{\UU}_R(s_i(\Delta_+))$, depending on whether $\alpha=\alpha_i$ or $\alpha=-\alpha_i$.

The following lemma will be crucial for us.
\begin{lemma}\label{lemma conjug root groups}
For any $\alpha\in\Delta_+$ and any $w\in W$ such that $w\alpha\in\Delta_+$, one has
$$\overline{w}\U^{ma}_{(\alpha)}\overline{w}\thinspace\inv=\U^{ma}_{(w\alpha)}.$$
\end{lemma}
\begin{proof}
For $\alpha$ a real root, this is \cite[3.11]{Rousseau}. In any case, this amounts to showing that, whenever $s_i\in W$ is such that $s_i(\alpha)\in\Delta_+$, one has $$\overline{s_i}\cdot([\exp]x)\cdot\overline{s_i}\thinspace\inv=[\exp](s_i^*x)$$ for any homogenous $x\in\oplus_{n\geq 1}{\g_{n\alpha R}}$, with $R$ an arbitrary ring. This last statement readily follows from the definition of the semi-direct product defining $\B_i^{ma+}$.
\end{proof}

\subsubsection*{The group scheme $\G^{pma}$}
The {\bf Mathieu--Rousseau completion} $\G^{pma}$ of $\G$ is then defined as some amalgamated product of the minimal parabolics $\B_i^{ma+}$, $i\in I$ (see \cite[3.6]{Rousseau}). Over the field $k$, the identification of $\{\U_{\alpha}(k) \ | \ \alpha\in\Delta^{\re}\}$ with the root group datum of $\G(k)$ (as well as the identification of the tori $\T(k)$ of $\G(k)$ and $\G^{pma}(k)$) induces an injection of $\G(k)$ in $\G^{pma}(k)$ (see \cite[Proposition~3.13]{Rousseau}). The Borel subgroup $\B^{ma+}(k)=\T(k)\ltimes\U^{ma+}(k)$ and $\N(k)$ form a BN-pair for $\G^{pma}(k)$ with associated building the positive building of $\G(k)$ (see \cite[3.16]{Rousseau}).

The topology on $\G^{pma}(k)$ is given as follows. For each $n\in\NN$, set $\U^{ma}_n:=\U^{ma}_{\Psi(n)}$, where $\Psi(n)=\{\alpha\in\Delta^+ \ | \ \height(\alpha)\geq n\}$. 
\begin{lemma}[{\cite[6.3.6]{Rousseau}}] \label{lemma topoMathieu}
$\G^{pma}(k)$ is a complete (Hausdorff) topological group with basis of neighbourhoods of the identity the subgroups $\U^{ma}_n(k)$, $n\in\NN$.
\end{lemma}

\subsubsection*{Comparison with the Rémy--Ronan completion}
Recall from the introduction the continuous homomorphism $\varphi\co \overline{\G(k)}\to\G^{rr}(k)$, where $\overline{\G(k)}$ denotes the closure of $\G(k)$ in $\G^{pma}(k)$. Write also $\overline{\U^+(k)}$ for the closure of $\U^+(k)$ in $\U^{ma+}(k)$.

\begin{lemma}[{\cite[6.3.5]{Rousseau}}]\label{lemma phi}
Assume that the field $k$ is finite. Then the restriction of $\varphi$ to $\overline{\U^+(k)}$ is surjective onto $\U^{rr+}(k)$.
\end{lemma}

\section{Coxeter groups and root systems}\label{section W and Delta}
In this section, we prepare the ground for the proof of Theorem~\ref{mainthmintro} by establishing several results which concern the Coxeter group $W$ and the set of roots $\Delta$. Basics on these two topics are covered in \cite[Chapters 1--3]{BrownAbr} and \cite[Chapters 1--5]{Kac}, respectively. 

Throughout this section, we let $\Sigma=\Sigma(W,S)$ denote the Coxeter complex of $W$. Also, we let $C_0$ be the fundamental chamber of $\Sigma$. Finally, with the exception of Lemma~\ref{lemma matrix Coxeter element} below where no particular assumption on $W$ is made, we will always assume that $W$ is infinite irreducible. Note that this is equivalent to saying that $A$ is indecomposable of non-finite type.

\begin{lemma}\label{lemma matrix Coxeter element}
Let $w=s_1\dots s_n$ be a Coxeter element of $W$. Let $A=A_1+A_2$ be the unique decomposition of $A$ as a sum of matrices $A_1, A_2$ such that $A_1$ (respectively, $A_2$) is an upper (respectively, lower) triangular matrix with $1$'s on the diagonal. Then the matrix of $w$ in the basis $\{\alpha_1,\dots,\alpha_n\}$ of simple roots is $-A_1\inv A_2=I_n-A_1\inv A$.
\end{lemma}
\begin{proof}
For a certain property $\Prop$ of two integer variables $i,j$ (e.g. $\Prop(i,j)\equiv j\leq i$), we introduce for short the Kronecker symbol $\delta_{\Prop(i,j)}$ taking value $1$ if $\Prop(i,j)$ is satisfied and $0$ otherwise.

Let $B=(b_{ij})$ denote the matrix of $w$ in the basis $\{\alpha_1,\dots,\alpha_n\}$. Thus, $b_{ij}$ is the coefficient of $\alpha_i$ in the expression of $s_1\dots s_n\alpha_j$ as a linear combination of the simple roots, which we will write for short as $[s_1\dots s_n\alpha_j]_i$. Thus $b_{ij}=[s_1\dots s_n\alpha_j]_i=[s_{i}\dots s_n\alpha_j]_i$. Note that
\begin{align*}
s_{i+1}\dots s_n\alpha_j=\sum_{k=i+1}^{n}{[s_{i+1}\dots s_n\alpha_j]_k}\alpha_k+\delta_{j\leq i}\alpha_j=\sum_{k=i+1}^{n}{b_{kj}}\alpha_k+\delta_{j\leq i}\alpha_j.
\end{align*}
Whence
\begin{align*}
b_{ij}&=[s_i(\sum_{k=i+1}^{n}{b_{kj}}\alpha_k+\delta_{j\leq i}\alpha_j)]_i=-\sum_{k=i+1}^{n}{a_{ik}b_{kj}}-\delta_{j\leq i}a_{ij}+\delta_{i=j}\\
&= (-\sum_{k=1}^{n}{(A_1)_{ik}b_{kj}}+b_{ij})+(\delta_{j>i}a_{ij}-a_{ij})+\delta_{i=j}\\
&= -\sum_{k=1}^{n}{(A_1)_{ik}b_{kj}}+b_{ij}-a_{ij}+\sum_{k=1}^{n}{(A_1)_{ik}(I_n)_{kj}}.
\end{align*}
Thus $A=-A_1B+A_1$, so that $B=-A_1\inv A_2$, as desired.
\end{proof}

For $\omega\in W$ and $\alpha\in\Delta_+$, define the function $f^{\omega}_{\alpha}\co\ZZ\to\{\pm 1\}:k\mapsto\sign(\omega^k\alpha)$, where $\sign(\Delta_{\pm})=\pm 1$.

\begin{lemma}\label{lemma dichotomy psi}
Let $\omega\in W$ be such that $\ell(\omega^l)=|l|\ell(\omega)$ for all $l\in\ZZ$.  Then $f^{\omega}_{\alpha}$ is monotonic for all $\alpha\in\Delta_+$.
\end{lemma}
\begin{proof}
Let $\omega\in W$ be such that $\ell(\omega^l)=|l|\ell(\omega)$ for all $l\in\ZZ$ and let $\omega=t_1t_2\dots t_k$ be a reduced expression for $\omega$, where $t_j\in S$ for all $j\in\{1,\dots,k\}$.
Let $\alpha\in\Delta_+$ and assume that $f^{\omega}_{\alpha}$ is not constant. Then $\alpha$ is a real root because $W.\Delta_+^{\im}=\Delta_+^{\im}$. 
Let $k_{\alpha}\in\ZZ^*$ be minimal (in absolute value) so that $f^{\omega}_{\alpha}(k_{\alpha})=-1$. We deal with the case when $k_{\alpha}>0$; the same proof applies for $k_{\alpha}<0$ by replacing $\omega$ with its inverse.  We have to show that $\omega^l\alpha\in\Delta_-$ if and only if $l\geq k_{\alpha}$. 

Let $\beta:=\omega^{k_{\alpha}-1}\alpha$. Thus $\beta\in\Delta_+^{\re}$ and $\omega\beta\in\Delta_-^{\re}$. It follows that there is some $i\in \{1,\dots,k\}$ such that $\beta=t_kt_{k-1}\dots t_{i+1}\alpha_{t_i}$. In other words, $\beta$ is one of the $n$ positive roots whose wall $\partial\beta$ in the Coxeter complex $\Sigma$ of $W$ separates the fundamental chamber $C_0$ from $\omega\inv C_0$. We want to show that $\omega^l\beta\in\Delta_-$ if and only if $l\geq 1$.

Assume first for a contradiction that there is some $l\geq 1$ such that $\omega^{l+1}\beta\in\Delta_+$, that is, $\omega^{l+1}\beta$ contains $C_0$. Since $\omega^{l+1}\beta$ contains $\omega^{l+1}C_0$ but not $\omega^lC_0$, its wall $\omega^{l+1}\partial\beta$ separates $\omega^{l}C_0$ from $\omega^{l+1}C_0$ and $C_0$. In particular, any gallery from $C_0$ to $\omega^{l+1}C_0$ going through $\omega^lC_0$ cannot be minimal. This contradicts the assumption that $\ell(\omega^l)=|l|\ell(\omega)$ for all $l\in\ZZ$ since this implies that the product of $l+1$ copies of $t_1\dots t_k$ is a reduced expression for $\omega^{l+1}$. 

Assume next for a contradiction that there is some $l\geq 1$ such that $\omega^{-l}\beta\in\Delta_-$. Then as before, $\omega^{-l}\partial\beta$ separates $\omega^{-l}C_0$ from $\omega^{-l-1}C_0$ and $C_0$. Again, this implies that any gallery from $C_0$ to $\omega^{-l-1}C_0$ going through $\omega^{-l}C_0$ cannot be minimal, yielding the desired contradiction. 
\end{proof}

\begin{corollary}\label{cor w monotonic}
Let $w=s_1\dots s_n$ be a Coxeter element of $W$. Then $f^{w}_{\alpha}$ is monotonic for all $\alpha\in\Delta_+$.
\end{corollary}
\begin{proof}
As $\ell(w^l)=|l|\ell(w)$ for all $l\in\ZZ$ by the main result of \cite{wkreduced}, this readily follows from Lemma~\ref{lemma dichotomy psi}.
\end{proof}

\begin{lemma}\label{lemma no fixed root for w}
Assume that $A$ is of indefinite type. Let $w=s_1\dots s_n$ be a Coxeter element of $W$, and let $\alpha\in\Delta_+$. Then $w^l\alpha\neq\alpha$ for all nonzero integer $l$.
\end{lemma}
\begin{proof}
Assume for a contradiction that $w^k\alpha=\alpha$ for some $k\in\NN^*$. It then follows from Corollary~\ref{cor w monotonic} that $w^i\alpha\in\Delta_+$ for all $i\in\{0,\dots,k-1\}$. Viewing $w$ as an automorphism of the root lattice, we get that $$(w-\Id)(w^{k-1}+\dots+w+\Id)\alpha=0.$$
Moreover, $\beta:=(w^{k-1}+\dots+w+\Id)\alpha$ is a sum of positive roots, and hence can be viewed as a nonzero vector of $\RR^n$ with nonnegative entries. Recall from Lemma~\ref{lemma matrix Coxeter element} that $w$ is represented by the matrix $-A_1\inv A_2$. Thus, multiplying the above equality by $-A_1$, we get that $A\beta=0$. Since $A$ is indecomposable of indefinite type, this gives the desired contradiction by \cite[Theorem 4.3]{Kac}.
\end{proof}

\begin{lemma}\label{lemma root groups contracting}
Let $\omega\in W$ and $\alpha\in\Delta_+$ be such that $\omega^l\alpha\neq\alpha$ for all positive integer $l$. Then  $|\height(\omega^{l}\alpha)|$ goes to infinity as $l$ goes to infinity.
\end{lemma}
\begin{proof}
If $|\height(\omega^{l}\alpha)|$ were bounded as $l$ goes to infinity, the set of roots $\{\omega^{l}\alpha \ | \ l\in\NN\}$ would be finite, and so there would exist an $l\in\NN^*$ such that $\omega^{l}\alpha=\alpha$, a contradiction. 
\end{proof}

\begin{lemma}\label{lemma height to infinity}
Assume that $A$ is of indefinite type. Let $w=s_1\dots s_n$ be a Coxeter element of $W$, and let $\alpha\in\Delta_+$. Then there exists some $\epsilon\in\{\pm\}$ such that $w^{\epsilon l}\alpha\in\Delta_+$ for all $l\in\NN$. Moreover, $\height(w^{\epsilon l}\alpha)$ goes to infinity as $l$ goes to infinity.
\end{lemma}
\begin{proof}
The existence of $\epsilon$ readily follows from Corollary~\ref{cor w monotonic}, while the second statement is a consequence of Lemmas~\ref{lemma no fixed root for w} and \ref{lemma root groups contracting}.
\end{proof}

\section{Contraction groups}\label{section contraction}
In this section, we make use of the results proven so far to establish, under suitable hypotheses, that the subgroups $\U^{ma+}(k)$ of $\G^{pma}(k)$ and $\U^{rr+}(k)$ of $\G^{rr}(k)$ are contracted. 
Throughout this section, $W$ is assumed to be infinite irreducible, and we fix some Coxeter element $w=s_1\dots s_n$ of $W$.

\begin{lemma}\label{lemma dense union}
Let $\Psi_1\subseteq\Psi_2\subseteq\dots\subseteq\Delta_+$ be an increasing sequence of closed subsets of $\Delta_+$ and set $\Psi=\bigcup_{i=1}^{\infty}{\Psi_i}$. Then the corresponding increasing union of subgroups $\bigcup_{i=1}^{\infty}{\U^{ma}_{\Psi_i}(k)}$ is dense in $\U^{ma}_{\Psi}(k)$.
\end{lemma}
\begin{proof}
This readily follows from Proposition~\ref{prop formal sum Rousseau}.
\end{proof}

\begin{prop}\label{prop conjugation root groups}
Let $\Psi\subseteq\Delta_+$ be closed. Let $\omega\in W$ be such that $\omega\Psi\subseteq \Delta_+$. Then $\overline{\omega}\U^{ma}_\Psi\overline{\omega}\thinspace\inv=\U^{ma}_{\omega\Psi}$.
\end{prop}
\begin{proof}
For a positive root $\alpha\in\Delta_+$, consider the root group $\U^{ma}_{(\alpha)}$ as in Section~\ref{subsection outline MR completion}.
Let also $\Psi$ and $\omega$ be as in the statement of the lemma. By Lemma~\ref{lemma conjug root groups}, we know that $$\overline{\omega}\la\U^{ma}_{(\alpha)} \ | \ \alpha\in\Psi\ra\overline{\omega}\thinspace\inv=\la\U^{ma}_{(\omega\alpha)} \ | \ \alpha\in\Psi\ra.$$
Passing to the closures, it follows from Lemma~\ref{lemma dense union} that $\overline{\omega}\U^{ma}_\Psi\overline{\omega}\thinspace\inv=\U^{ma}_{\omega\Psi}$, as desired.
\end{proof}

\begin{lemma}\label{lemma decomposition}
Let $\Psi\subseteq\Delta_+$ be the set of positive roots $\alpha$ such that $w^l\alpha\in\Delta_+$ for all $l\in\NN$. Then both $\Psi$ and $\Delta_+\setminus\Psi$ are closed. In particular, one has a unique decomposition $\U^{ma+}=\U^{ma}_{\Psi}.\U^{ma}_{\Delta_+\setminus\Psi}$.
\end{lemma}
\begin{proof}
Clearly, $\Psi$ is closed. Let now $\alpha,\beta\in\Delta_+\setminus\Psi$ be such that $\alpha+\beta\in\Delta$. Thus there exist some positive integers $l_1,l_2$ such that $w^{l_1}\alpha\in\Delta_-$ and $w^{l_2}\beta\in\Delta_-$. Then $w^l(\alpha+\beta)\in\Delta_-$ for all $l\geq\max\{l_1,l_2\}$ by Corollary~\ref{cor w monotonic} and hence $\alpha+\beta\in\Delta_+\setminus\Psi$. Thus $\Delta_+\setminus\Psi$ is closed, as desired.
The second statement follows from Lemma~\ref{lemma decomp Rousseau}.
\end{proof}

\begin{remark}\label{remark dense union}
Let $\Psi\subseteq\Delta_+$ be as in Lemma~\ref{lemma decomposition}. Put an arbitrary order on $\Delta_+$. This yields enumerations $\Psi=\{\beta_1,\beta_2,\dots\}$ and $\Delta_+\setminus\Psi=\{\alpha_1,\alpha_2,\dots\}$. For each $i\in\NN^*$, we let $\Psi_i$ (respectively, $\Phi_i$) denote the closure in $\Delta_+$ of $\{\beta_1,\dots,\beta_i\}$ (respectively, of $\{\alpha_1,\dots,\alpha_i\}$). It follows from Lemma~\ref{lemma decomposition} that $\Psi=\bigcup_{i=1}^{\infty}{\Psi_i}$ and that $\Delta_+\setminus\Psi=\bigcup_{i=1}^{\infty}{\Phi_i}$.
\end{remark}

\begin{lemma}\label{lemma height functor}
Fix $i\in\NN^*$, and let $\Psi_i,\Phi_i\subseteq\Delta_+$ be as in Remark~\ref{remark dense union}. Assume that $A$ is of indefinite type. Then there exists a sequence of positive integers $(n_l)_{l\in\NN}$ going to infinity as $l$ goes to infinity, such that $\overline{w}^{\thinspace l}\U^{ma}_{\Psi_i}\overline{w}^{\thinspace -l}\subseteq \U^{ma}_{n_l}$ and $\overline{w}^{\thinspace -l}\U^{ma}_{\Phi_i}\overline{w}^{\thinspace l}\subseteq \U^{ma}_{n_l}$ for all $l\in\NN$.
\end{lemma}
\begin{proof}
Let $\alpha_j,\beta_j\in\Delta_+$ be as in Remark~\ref{remark dense union}. By Lemma~\ref{lemma height to infinity} together with Corollary~\ref{cor w monotonic}, one can find for each $j\in\{1,\dots,i\}$ sequences of positive integers $(m^j_l)_{l\in\NN}$ and $(n^j_l)_{l\in\NN}$ going to infinity as $l$ goes to infinity, such that $\height(w^{-l}\alpha_j)\geq m^j_l$ and $\height(w^{l}\beta_j)\geq n^j_l$ for all $l\in\NN$. For each $l\in\NN$, set $n_l=\min\{m_j^l,n_j^l \ | \ 1\leq j\leq i\}$. Then the sequence $(n_l)_{l\in\NN}$ goes to infinity as $l$ goes to infinity. Moreover, $\height(\alpha)\geq n_l$ for all $\alpha\in w^{-l}\Phi_i$ and $\height(\beta)\geq n_l$ for all $\beta\in w^l\Psi_i$.  
The conclusion then follows from Proposition~\ref{prop conjugation root groups}.
\end{proof}

\begin{theorem}\label{thm contracting}
Let $a=\overline{w}\in \G(k)\subseteq\G^{pma}(k)$, and let $\Psi,\Psi_i,\Phi_i$ be as in Remark~\ref{remark dense union}. Assume that $A$ is of indefinite type. Then the following hold.
\begin{enumerate}
\item
$\U^{ma}_{\Psi_i}(k)\subseteq \con(a)$ and $\U^{ma}_{\Phi_i}(k)\subseteq \con(a\inv)$ for all $i\in\NN^*$.
\item
$\U^{ma}_{\Psi}(k)\subseteq \overline{\con(a)}$ and $\U^{ma}_{\Delta_+\setminus\Psi}(k)\subseteq \overline{\con(a\inv)}$.
\item
$\U^{ma+}(k)\subseteq\la \overline{\con(a)}\cup \overline{\con(a\inv)}\ra$.
\end{enumerate}
\end{theorem}
\begin{proof}
The first statement follows from Lemma~\ref{lemma height functor}. The second statement is a consequence of the first together with Lemma~\ref{lemma dense union}. The third statement follows from the second together with Lemma~\ref{lemma decomposition}.
\end{proof}

Recall from Lemma~\ref{lemma phi} that $\varphi(\overline{\U^+(k)})=\U^{rr+}(k)$ whenever $k$ is finite.

\begin{lemma}\label{lemma Urr contracting}
Let $a=\overline{w}\in \G(k)\subseteq\G^{rr}(k)$. Assume that $A$ is of indefinite type and that $\varphi(\overline{\U^+(k)})=\U^{rr+}(k)$ (e.g. $k$ finite). Then $\U^{rr+}(k)\subseteq \la \overline{\con(a)}\cup \overline{\con(a\inv)}\ra$.
\end{lemma}
\begin{proof}
We know from Theorem~\ref{thm contracting}~(3) that $\overline{\U^+(k)}\subseteq\la \overline{\con(a)}\cup \overline{\con(a\inv)}\ra$. Applying $\varphi$ then yields the desired inclusion since $\varphi(\overline{\con(a^{\pm 1})})\subseteq \overline{\con(\varphi(a)^{\pm 1})}=\overline{\con(a^{\pm 1})}$ by continuity of $\varphi$.
\end{proof}

\subsection*{Proof of Theorem~\ref{thmBintro} }
The first statement is Proposition~\ref{prop conjugation root groups} and the third follows from Theorem~\ref{thm contracting}~(1). The second statement is a consequence of the first together with Lemmas~\ref{lemma topoMathieu} and \ref{lemma root groups contracting}.\hspace{\fill}$\Box$

\section{Consequences of Theorem~\ref{thmBintro}}\label{section proof corollaries}
Before we give the proof of Theorems~\ref{mainthmintro} and \ref{thmintro simple pma} in the next section, we examine the consequences, stated in the introduction, of Theorem~\ref{thmBintro}.
More precisely, we will make use of the following lemma. Recall from \cite[Section~3]{WilNub} the definition of the {\bf nub}\index{Nub of an automorphism} of an automorphism $\alpha$ of a totally disconnected locally compact group $G$. It possesses many equivalent definitions (see \cite[Theorem 4.12]{WilNub}), and given an element $a\in G$ (viewed as a conjugation automorphism), it can be characterised as $\nub(a)=\overline{\con(a)}\cap \overline{\con(a\inv)}$ (see \cite[Remark 3.3 (b) and (d)]{WilNub}).

\begin{lemma}\label{lemma conseq them contracting}
Let $G=\G_{A}^{pma}(\FF_q)$ be a complete Kac--Moody group of simply connected type over a finite field $\FF_q$, with indecomposable generalised Cartan matrix $A$ of indefinite type. Let $U^{im+}=\U^{ma}_{\Delta_+^{\im}}(\FF_q)$ denote its positive imaginary subgroup, let $w\in W=W(A)$ denote a Coxeter element of $W$, and set $a:=\overline{w}\in\N(\FF_q)$. Then $$U^{im+}\subseteq\nub(a)=\overline{\con(a)}\cap \overline{\con(a\inv)}.$$
\end{lemma}
\begin{proof}
Notice that Lemma~\ref{lemma height functor} remains valid if one replaces $\Psi$ by its (closed) subset $\Delta_+^{\im}$ and $w$ by $w\inv$. As in the proof of the second statement of Theorem~\ref{thm contracting}, Lemma~\ref{lemma dense union} then allows to conclude.
\end{proof}

To establish Theorem~\ref{thmintro closed contraction groups}, we need one more technical lemma regarding contraction groups, whose proof is an adaptation of the proof of Proposition~2.1 in \cite{Wang}.
\begin{lemma}\label{lemma compact contracted}
Let $G$ be a locally compact group, let $a$ be an element of $G$, and let $Q$ be a compact subset of $G$ such that $Q\subseteq\con(a)$. Then $Q$ is uniformly contracted by $a$, that is, for every open neighbourhood $U$ of the identity one has $a^nQa^{-n}\subset U$ for all large enough $n$. 
\end{lemma}
\begin{proof}
Fix an open neighbourhood $U$ of the identity, and let $V$ be a compact neighbourhood of the identity such that $V^2\subset U$. By hypothesis, for all $x\in Q$ there exists an $N_x$ such that $a^{n}xa^{-n}\in V$ for all $n\geq N_x$. In other words, $$Q\subset \bigcup_{N\geq 0}{\bigcap_{n\geq N}{a^{-n}Va^{n}}}.$$
Note that the sets $C_N=\bigcap_{n\geq N}{a^{-n}Va^{n}}$ form an ascending chain of compact sets. It follows from Baire theorem that $Q\cap C_N$ has nonempty interior in $Q$ for a large enough $N$.

By compactness of $Q$, one then finds a finite subset $F$ of $Q$ such that $$Q\subset F.C_N.$$ Since $F$ is finite and contained in $\con(a)$, we know that $a^nFa^{-n}\subset V$ for all large enough $n$. Moreover, by construction, $a^nC_Na^{-n}\subset V$ for $n\geq N$, and hence
$$a^n Q a^{-n} = (a^n F a^{-n}).(a^n C_N a^{-n}) \subset V^2 \subset U$$ for all large enough $n$, as desired.
\end{proof}

\subsection*{Proof of Theorem~\ref{thmintro closed contraction groups}}
Let $A$ denote an $n\times n$ generalised Cartan matrix of indecomposable indefinite type, let $W=W(A)$ be the associated Weyl group, and let $w=s_1\dots s_n$ denote a Coxeter element of $W$. Set $a:=\overline{w}\in\N(\FF_q)$. It then follows from Lemma~\ref{lemma conseq them contracting} that $$U^{ma+}_{im}:=\U^{ma}_{\Delta_+^{\im}}(\FF_q)\subseteq \overline{\con(a)}\quad\textrm{in}\quad\G_A^{pma}(\FF_q)$$ and that $$U^{rr+}_{im}:=\varphi(\U^{ma}_{\Delta_+^{\im}}(\FF_q)\cap \overline{\U^+(\FF_q)})\subseteq \overline{\con(a)}\quad\textrm{in}\quad\G^{rr}_A(\FF_q).$$

Since $U^{ma+}_{im}$ is closed in $\U^{ma+}(\FF_q)$ which is compact (see \cite[6.3]{Rousseau}), both the groups $U^{ma+}_{im}$ and $U^{rr+}_{im}$ are compact. Moreover, they are normalised by $a$ by Proposition~\ref{prop conjugation root groups}. Hence they cannot be contracted by $a$ because of Lemma~\ref{lemma compact contracted}, since by assumption $U^{rr+}_{im}$ is nontrivial. In particular, $\con(a)\neq \overline{\con(a)}$ and hence $\con(a)$ cannot be closed.

Note that one could also directly use the fact that $\con(a)$ is closed if and only if $\nub(a)=\{1\}$ (see \cite[Remark 3.3 (b)]{WilNub}) together with Lemma~\ref{lemma conseq them contracting}. We preferred however to present a more elementary proof as well, as Lemma~\ref{lemma compact contracted} will be used anyway in the proof of Theorem~\ref{thmintro isom KM} below. \hspace{\fill}$\Box$

\medskip

To prove Theorem~\ref{thmintro isom KM}, we need two additional technical lemmas.
The first lemma is an adaptation of \cite[Lemma 4.3]{thesePE}. 
\begin{lemma}\label{lemma exceptional isom}
Let $k$ be a finite field of order $q$, and consider the generalised Cartan matrices $A=(\begin{smallmatrix}2&-m \\ -n&2\end{smallmatrix})$ and $A'=(\begin{smallmatrix}2&-m' \\ -n'&2\end{smallmatrix})$ such that $m,m',n,n'\geq 2$. Assume moreover that $m\equiv m' \ (\modulo q-1)$ and $n\equiv n' \ (\modulo q-1)$. Then the minimal Kac--Moody groups $\G_A(k)$ and $\G_{A'}(k)$ of simply connected type are isomorphic.
\end{lemma}
\begin{proof}
As the Weyl groups of $A$ and $A'$ are isomorphic (to the infinite dihedral group), one can identify the corresponding sets of real roots. Moreover, as noted in the proof of \cite[Lemma 4.3]{thesePE}, the commutation relations between root groups corresponding to prenilpotent pairs of roots are trivial in $\G_A(k)$ (resp. $\G_{A'}(k)$). In particular, one can identify the Steinberg functors of $\G_A$ and $\G_{A'}$.

Recall from Section~\ref{subsection outline MR completion} (and in the notations of that section) that the torus $\T_{\Lambda}(k)$ of $\G_A(k)$ is generated by $\{r^{\alpha_i^{\vee}} \ | \ r\in k^\times, \ i=1,2\}$, and similarly for the torus $\T_{\Lambda'}(k)$ of $\G_{A'}(k)$.
This yields identifications of $\T_{\Lambda}(k)$ and $\T_{\Lambda'}(k)$. As $r^m=r^{m'}$ and $r^n=r^{n'}$ for all $r\in k$, it then follows from the above identifications that $\G_A(k)$ and $\G_{A'}(k)$ admit the same Steinberg type presentation (see \cite[\S 3.6]{Tits87}), as desired.
\end{proof}

\begin{lemma}\label{lemma non degeneracy action}
Let $k$ be an arbitrary field with $\charact k\neq 2$. Let $G=\G^{pma}(k)$ be the Mathieu--Rousseau completion associated with a $2\times 2$ generalised Cartan matrix $A=(\begin{smallmatrix}2&-m \\ -m&2\end{smallmatrix})$ for some $m,n>2$. Then the imaginary subgroup $U^{im+}=\U^{ma}_{\Delta^{\im}_+}(k)$ of $G$ is not contained in the kernel $Z'(G)$ of the action of $G$ on its associated building.
\end{lemma}
\begin{proof}
Let $p$ denote the characteristic of $k$. Thus $p=0$ or $p\geq 3$. Assume for a contradiction that $U^{im+}$ is contained in $Z'(G)$. 

Note first that $U^{im+}=\bigcap_{w\in W}{\overline{w}\U^{ma+}(k)\overline{w}\thinspace\inv}$. Indeed, as $\Delta^{\im}_+$ is $W$-stable, the inclusion $\subseteq$ is clear from Proposition~\ref{prop formal sum Rousseau} and Lemma~\ref{lemma conjug root groups}. Assume now for a contradiction that there is some $u\in\bigcap_{w\in W}{\overline{w}\U^{ma+}(k)\overline{w}\thinspace\inv}$ that is not in $U^{im+}$. Write $u$ as a product $u=\prod_{x\in\BB_{\Delta_+}}{[\exp]\lambda_xx}$ as in Proposition~\ref{prop formal sum Rousseau}, and let $\Phi_u$ be the set of positive real roots $\beta$ such that $\lambda_x\neq 0$ for $x\in\BB_{\beta}$. Thus $\Phi_u$ is nonempty. Choose $\beta\in\Phi_u$ and $v\in W$ such that $-v\beta$ is a simple root and $v\beta'\in\Delta_+$ for all $\beta'\in\Phi_u\setminus\{\beta\}$. Then by Lemma~\ref{lemma conjug root groups}, the element $\overline{v}$ conjugates $u$ outside $\U^{ma+}(k)$, yielding the desired contradiction. 

As $Z'(G)=Z(G).(Z'(G)\cap \U^{ma+}(k))$ and as $Z'(G)\cap \U^{ma+}(k)$ is normal in $G$ by \cite[Proposition 6.4]{Rousseau}, where $Z(G)$ denotes the center of $G$, we deduce that $U^{im+}=Z'(G)\cap \U^{ma+}(k)$ is normal in $G$.

Recall the notations from Section~\ref{subsection outline MR completion}. In particular, $e_1, e_2$ and $f_1, f_2$ denote the Chevalley generators of the Kac--Moody algebra $\g$ with generalised Cartan matrix $A$, and $\alpha_1,\alpha_2$ (resp. $\alpha_1^{\vee},\alpha_2^{\vee}$) are the corresponding simple roots (resp. coroots). Recall also the definition of the $\ZZ$-form $\UU$ of $\UU_{\CC}(\g)$, as well as the Lie algebras $\g_{\ZZ}=\g\cap\UU$ and $\g_{k}=\g_{\ZZ}\otimes_{\ZZ}k$. Finally, for each real root $\gamma\in\Delta^{\re}$, choose as before a $\ZZ$-basis element $e_{\gamma}$ of $\g_{\gamma\ZZ}$.

We will show that there exist an imaginary root $\delta\in\Delta_+^{\im}$, a simple root $\alpha_i$, and a nonzero element $x\in\g_{\delta k}$ such that $\delta-\alpha_i\in\Delta_+^{\re}$ and $\ad(f_i)x$ is nonzero in $\g_{k}$. Recalling from Section~\ref{subsection outline MR completion} the definition of the semi-direct product $\B_i^{ma+}=\A^{\Lambda}_{i}\ltimes \U^{ma}_{\Delta_+\setminus\{\alpha_i\}}$, this will imply that the root group $\U_{-\alpha_i}(k)$ conjugates the imaginary root group $\U^{ma}_{(\thinspace\delta)}(k)$ outside $U^{im+}$, so that $U^{im+}$ cannot be normal in $G$, yielding the desired contradiction.

Assume first that $m$ is not a multiple of $p$. As $m,n\geq 3$, we know that $\delta:=\alpha_1+\alpha_2$ is an imaginary root (see \cite[Lemma 5.3]{Kac}) and that $x:=[e_1,e_2]\in \g_{\ZZ}$ is nonzero. Moreover, $\ad(f_1)x=me_2$ is nonzero in $\g_{k}$ since $m$ is not a multiple of $p$, as desired.

Assume next that $m$ is a multiple of $p$. Set $\gamma:=s_1(\alpha_2)=\alpha_2+m\alpha_1\in\Delta^{\re}_+$. Then again $\delta:=\alpha_2+\gamma\in\Delta^{\im}_+$ since $\la \delta,\alpha_1^{\vee}\ra= 0$ and $\la \delta,\alpha_2^{\vee}\ra= 4-mn<0$. Set $x:=[e_2,e_{\gamma}]\in\g_{\ZZ}$. Note that $\ad(f_2)e_{\gamma}=0$ since $\gamma-\alpha_2=m\alpha_1\notin\Delta$. As $p\neq 2$, we deduce that $\ad(f_2)x=(2-mn)e_{\gamma}$ is nonzero in $\g_{k}$, as desired.
\end{proof}

\subsection*{Proof of Theorem~\ref{thmintro isom KM}}
It follows from Lemma~\ref{lemma exceptional isom} that the minimal Kac--Moody group $G_1=\G_{A_1}(\FF_q)$ over $\FF_q$ of simply connected type with generalised Cartan matrix $A_1=(\begin{smallmatrix}2&-2 \\ -2&2\end{smallmatrix})$ (hence of affine type) is isomorphic to any minimal Kac--Moody group $G_2=\G_{A_2}(\FF_q)$ over $\FF_q$ of simply connected type with generalised Cartan matrix $A_2=(\begin{smallmatrix}2&-m \\ -n&2\end{smallmatrix})$ for some $m,n>2$ (hence of indefinite type) with $m\equiv n\equiv 2 \ (\modulo q-1)$. We fix such a group $G_2$.

For $i=1,2$ set $\widehat{G}_i:=\G^{pma}_{A_i}(\FF_q)$ and let $Z_i'$ denote the kernel of the action of $\widehat{G}_i$ on its associated building. Assume for a contradiction that there is an isomorphism $\psi\co\widehat{G}_1\to\widehat{G}_2$ of topological groups. As noticed in \cite[Remarque 6.20 (4)]{Rousseau}, the quotient $\widehat{G}_1/Z_1'$ is a simple algebraic group over the local field $\FF_q(\!(t)\!)$. In particular, all the contraction groups of $\widehat{G}_1/Z_1'$ are closed. Moreover, $\psi(Z_1')$ is the unique maximal proper normal subgroup of  $\widehat{G}_2$, and it is compact. It follows that $\psi(Z_1')=Z_2'$, for otherwise by Tits' lemma (see \cite[Lemma 6.61]{BrownAbr}), the group $\widehat{G}_2$ would be compact, a contradiction. Hence $\psi$ induces an isomorphism of topological groups between $\widehat{G}_1/Z_1'$ and $\widehat{G}_2/Z_2'$, so that in particular all contraction groups of $\widehat{G}_2/Z_2'$ are closed. Let $\pi\co \widehat{G}_2\to \widehat{G}_2/Z_2'$ denote the canonical projection, and let $a$ be any element of $\widehat{G}_2$. Then
$$\pi(\overline{\con(a)})\subseteq\overline{\pi(\con(a))}\subseteq \overline{\con(\pi(a))}=\con(\pi(a)).$$
It follows from Lemma~\ref{lemma conseq them contracting} that the subgroup $U^{im+}:=\U^{ma}_{\Delta_+^{\im}}(\FF_q)$ of $\U^{ma}_{\Delta_+}(\FF_q)$ in $\widehat{G}_2$ is such that 
$$\pi(U^{im+})\subseteq \pi(\overline{\con(a)})\subseteq \con(\pi(a))$$ for a suitably chosen $a\in\widehat{G}_2$ normalising $U^{im+}$. Thus Lemma~\ref{lemma compact contracted} implies that $\pi(U^{im+})=\{1\}$, that is, $U^{im+}\subseteq Z_2'$. This contradicts Lemma~\ref{lemma non degeneracy action}, as desired.  \hspace{\fill}$\Box$

\subsection*{Proof of Corollary~\ref{corintro non density}}
Let $k=\FF_q$ with $\charact k\neq 2$, and let the generalised Cartan matrix $A$ be as in the statement of Corollary~\ref{corintro non density}. As we saw in the proof of Theorem~\ref{thmintro isom KM} above, the minimal Kac--Moody group $G_2=\G_{A_2}(\FF_q)$ (where we set $A_2=A$) is then isomorphic to the minimal Kac--Moody group $G_1=\G_{A_1}(\FF_q)$ over $\FF_q$ of simply connected type with generalised Cartan matrix $A_1=(\begin{smallmatrix}2&-2 \\ -2&2\end{smallmatrix})$, whereas the quotients $\widehat{G}_1/Z_1'$ and $\widehat{G}_2/Z_2'$ cannot be isomorphic as topological groups, where as before $\widehat{G}_i:=\G^{pma}_{A_i}(\FF_q)$ and $Z_i'$ is the kernel of the action of $\widehat{G}_i$ on its associated building ($i=1,2$).

Note that the isomorphism between $G_1$ and $G_2$ is the one provided by Lemma~\ref{lemma exceptional isom}, and it maps the twin BN-pair of $G_1$ to that of $G_2$. In particular, the Rémy--Ronan completions $\G^{rr}_{A_1}(\FF_q)$ of $G_1$ and $\G^{rr}_{A_2}(\FF_q)$ of $G_2$ are isomorphic as topological groups. 

Moreover, as $\charact k>2$, we know from \cite[Proposition~6.11]{Rousseau} that $G_1$ is dense in $\widehat{G}_1$. Assume now for a contradiction that $G_2$ is dense in $\widehat{G}_2$. Then the surjective continuous homomorphisms $\varphi_i\co \widehat{G}_i\to \G^{rr}_{A_i}(\FF_q)$ ($i=1,2$) induce isomorphisms of topological groups
$$\widehat{G}_1/Z_1'\cong \G^{rr}_{A_1}(\FF_q) \cong \G^{rr}_{A_2}(\FF_q) \cong \widehat{G}_2/Z_2',$$ yielding the desired contradiction.
\hspace{\fill}$\Box$

\section{Proof of Theorems~\ref{mainthmintro} and \ref{thmintro simple pma}}\label{section proof thmA}
We now let $k=\FF_q$ be a finite field, $A$ be an indecomposable generalised Cartan matrix of indefinite type, and we let $G$ be one of the complete Kac--Moody groups $\G_A^{rr}(\FF_q)$ or $\G_A^{pma}(\FF_q)$. We also set $U^+:=\U^{rr+}(\FF_q)$ or $U^+:=\U^{ma+}(\FF_q)$ accordingly. Then $G$ is a locally compact totally disconnected topological group, and $U^+$ is a compact open subgroup of $G$. Indeed, for $\G_A^{rr}(\FF_q)$, this follows from \cite[Proposition 1]{CaRe}; $\G_A^{pma}(\FF_q)$ is locally compact because $\U^{ma+}(\FF_q)$ is compact open by \cite[6.3]{Rousseau}, and it is totally disconnected because its filtration by the $\U^{ma}_{n}(\FF_q)$ is separated.

As mentioned in the introduction we first need to establish the topological simplicity of $\G_A^{pma}(\FF_q)$ in full generality.
\begin{prop}\label{prop pma topolsimple}
Assume that the generalised Cartan matrix $A$ is indecomposable of indefinite type. Then $\G_A^{pma}(\FF_q)/Z'(\G_A^{pma}(\FF_q))$ is topologically simple over any finite field $\FF_q$.
\end{prop}
\begin{proof}
Set $G:=\G_A^{pma}(\FF_q)$ and $Z':=Z'(\G_A^{pma}(\FF_q))$. It follows from \cite[Corollary~3.1]{CaMo} that $G$ possesses a closed cocompact normal subgroup $H$ containing $Z'$ and such that $H/Z'$ is topologically simple. It thus remains to see that in fact $H=G$. Let $\pi\co G\to G/H$ denote the canonical projection. Let also $w$ be a Coxeter element of $W$, and set $a:=\overline{w}\in\N(\FF_q)\subset G$. Since $G/H$ is compact and totally disconnected, its contraction groups are trivial (see e.g. the introduction of \cite{Titscore}). In particular,
$$\pi(\con(a^{\pm 1}))\subseteq \con(\pi(a^{\pm 1}))=\{1\},$$
and hence the closures of the contraction groups $\con(a)$ and $\con(a\inv)$ are contained in $\ker\pi=H$. It follows from Theorem~\ref{thm contracting} that $H$ contains $\U^{ma+}(\FF_q)$. But $G$ normalises $H$ and contains $\N(\FF_q)$, and hence $H$ also contains all real root groups. Therefore $H=G$, as desired.
\end{proof}

We can now give a common proof for Theorems~\ref{mainthmintro} and \ref{thmintro simple pma}:
\begin{theorem}\label{mainthmnotintro}
Assume that the generalised Cartan matrix $A$ is indecomposable of indefinite type, and let $G$ be one of the complete Kac--Moody groups $\G_A^{rr}(\FF_q)$ or $\G_A^{pma}(\FF_q)$. Then $G/Z'(G)$ is abstractly simple.
\end{theorem}
\begin{proof}
Set $U^+:=\U^{rr+}(\FF_q)$ or $U^+:=\U^{ma+}(\FF_q)$ so that $U^+\leq G$.
Let $K$ be a nontrivial normal subgroup of $G/Z'(G)$. Since $G/Z'(G)$ is topologically simple (see \cite[Proposition 11]{CaRe} for $\G_A^{rr}(\FF_q)$ and Proposition~\ref{prop pma topolsimple} for $\G_A^{pma}(\FF_q)$), $K$ must be dense in $G$. Since $G$ is locally compact and totally disconnected, it then follows from Theorem~\ref{thm contracting} and Lemma~\ref{lemma Urr contracting}, together with Theorem~\ref{thm CRW}, that $K$ contains $U^+$. Since $U^+$ is open, $K$ is open as well, and hence closed in $G$. Therefore $K=G$, as desired.
\end{proof}

\begin{remark}\label{rem blabla}
We remark that, although we made the assumption that the Kac--Moody group $\G(k)$ be of simply connected type (to get simplified statements), this of course does not have any impact on the simplicity results, and one might as well consider an arbitrary Kac--Moody root datum $\mathcal D$ and the Kac--Moody group $\G_{\mathcal D}(k)$. The essential difference is that, in general, $\G_{\mathcal D}(k)$ is not generated by its root subgroups anymore, and one then has to consider a subquotient of $\G^{pma}(k)$ (or else $\G^{rr}(k)$). More precisely, let $G$ be either $\G^{pma}(k)$ or $\G^{rr}(k)$, let $U^+$ be the corresponding subgroup $\U^{ma+}(k)$ or $\U^{rr+}(k)$, and let $G_{(1)}$ be the subgroup of $G$ generated by $U^+$ and by all root groups of $\G(k)$. Then $G_{(1)}$ is normal in $G$ (and $G=\T(k).G_{(1)}$), and what we proved is the abstract simplicity of $G_{(1)}/(Z'(G)\cap G_{(1)})$. 
\end{remark}

\bibliographystyle{amsalpha} 
\bibliography{simpleKM}

\end{document}